\documentclass[preprint,3p,12pt,times]{elsarticle}

\usepackage[T1]{fontenc}
\usepackage{lmodern}
\usepackage{microtype}
\usepackage{amsmath,amssymb,amsthm,mathtools}
\usepackage{enumitem}
\usepackage{aliascnt}
\usepackage{booktabs,tabularx,placeins}
\usepackage[unicode,colorlinks=true,linkcolor=blue!45!black,citecolor=blue!45!black,urlcolor=blue!45!black]{hyperref}
\usepackage[nameinlink,capitalise,noabbrev]{cleveref}

\biboptions{numbers,sort&compress}

\newtheorem{theorem}{Theorem}[section]

\newaliascnt{proposition}{theorem}
\newtheorem{proposition}[proposition]{Proposition}
\aliascntresetthe{proposition}

\newaliascnt{lemma}{theorem}
\newtheorem{lemma}[lemma]{Lemma}
\aliascntresetthe{lemma}

\newaliascnt{corollary}{theorem}
\newtheorem{corollary}[corollary]{Corollary}
\aliascntresetthe{corollary}

\theoremstyle{definition}
\newaliascnt{definition}{theorem}
\newtheorem{definition}[definition]{Definition}
\aliascntresetthe{definition}

\theoremstyle{remark}
\newaliascnt{remark}{theorem}
\newtheorem{remark}[remark]{Remark}
\aliascntresetthe{remark}

\newaliascnt{question}{theorem}
\newtheorem{question}[question]{Question}
\aliascntresetthe{question}

\newcommand{\N}{\mathbb N}
\newcommand{\Npos}{\mathbb N_{>0}}
\newcommand{\Acoord}{\mathcal A}
\newcommand{\J}{\mathbb J}
\newcommand{\RR}{\mathfrak R}
\newcommand{\DCPO}{\mathbf{DCPO}}
\newcommand{\Irr}{\operatorname{Irr}}
\newcommand{\down}{\mathord{{\downarrow}}}
\newcommand{\up}{\mathord{\uparrow}}
\newcommand{\supp}{\operatorname{supp}}
\newcommand{\Down}{\operatorname{Down}}
\newcommand{\cl}{\operatorname{cl}}
\newcommand{\rowprod}{\operatorname{row}_{\times}}
\newcommand{\rowframe}{\operatorname{row}_{F}}
\newcommand{\Pprod}{P_{\times}}
\newcommand{\Pframe}{P_F}
\newcommand{\Pstageprod}[1]{P^{\times}_{#1}}
\newcommand{\Pstageframe}[1]{P^{F}_{#1}}
\newcommand{\Tprod}{\mathcal T_{\times}}
\newcommand{\Tframe}{\mathcal T_F}
\newcommand{\hprod}{h_{\times}}
\newcommand{\hframe}{h_F}
\newcommand{\Fbot}{\bot_F}
\newcommand{\Ftop}{\top_F}

\newcommand{\Rtwo}{\mathcal R_2}

\newcommand{\two}{\mathbf 2}
\newcommand{\Atop}{\top_{\Acoord}}
\newcommand{\Abot}{\bot_{\Acoord}}
\newcommand{\Ichain}{\mathbb I}

\begin{document}

\begin{frontmatter}

\title{\texorpdfstring{$C_{\sigma}$}{C-sigma}-Unique Dcpos: Intrinsic Characterizations and Counterexamples}

\author[scu]{Yuxu Chen}
\ead{chenyuxu@scu.edu.cn}
\author[scu]{Xulong He}
\ead{hexulong@scu.edu.cn}

\affiliation[scu]{organization={School of Mathematics, Sichuan University},
  city={Chengdu},
  country={China}}

\begin{abstract}
A dcpo \(D\) is called \(C_{\sigma}\)-unique if for every dcpo \(Q\),
\(
\Gamma(D)\cong \Gamma(Q)
\) implies \(
D\cong Q,
\)
where \(\Gamma(D)\) denotes the lattice of Scott-closed subsets of \(D\).
We give an intrinsic characterization of \(C_{\sigma}\)-uniqueness in terms of Skula-density and Scott closure, identifying precisely when a proper sub-dcpo can preserve the lattice of Scott-closed sets.
Based on this characterization, we answer three open problems: (1) every power $\J^I$ of Johnstone's dcpo is $C_{\sigma}$-unique; (2) $C_{\sigma}$-uniqueness is not preserved by binary products, even when both factors and their product are sober; and (3) a sober countable complete lattice need not be $C_{\sigma}$-unique, even when it is a frame.
\end{abstract}

\begin{keyword}
Scott-closed-set lattice \sep $C_{\sigma}$-unique dcpo \sep
Skula-density \sep countable frame
\MSC[2020] 06B35 \sep 06A06 \sep 06D22 \sep 18A40 \sep 54D10
\end{keyword}

\end{frontmatter}

\section{Introduction}

For a dcpo $D$, let $\Gamma(D)$ denote the complete lattice of Scott-closed subsets, ordered by inclusion. The dcpo $D$ is $C_{\sigma}$-unique if
\[
  \Gamma(D)\cong\Gamma(Q)\quad\Longrightarrow\quad D\cong Q
\]
for every dcpo $Q$. This is the objectwise reconstruction problem initiated by Ho and Zhao \cite{HoZhao2009}. Its negative solution shows that non-isomorphic dcpos can have isomorphic Scott-closed-set lattices \cite{ho2018}. The task is therefore to identify when this lattice still determines the dcpo.

Zhao and Xu established positive results for quasicontinuous dcpos and Johnstone's dcpo, together with reconstruction criteria based on bounded sobriety and directed generation \cite{ZhaoXu2018,XuZhao2019}. Zhao subsequently isolated a number of open questions \cite{Zhao2022Survey,Zhao2025Johnstone}. Three of them are central here: whether the countable power $\J^{\N}$ of Johnstone's dcpo is $C_{\sigma}$-unique, whether binary products preserve $C_{\sigma}$-uniqueness, and whether every sober countable complete lattice is $C_{\sigma}$-unique.

We use the reflective hull $\mathcal R_2$ of the two-element chain in $\mathbf{DCPO}$. The recent work of He, Lyu, Chen and Kou proves that every dcpo $D$ has a two-element reflection $\lambda_D:D\to L_2D$ inducing $\Gamma(L_2D)\cong\Gamma(D)$, that $\mathcal R_2$ is $\Gamma$-faithful, and that every $C_{\sigma}$-unique dcpo belongs to $\Rtwo$ \cite{HeLyuChenKou2026}. 

Our first result gives an intrinsic characterization of \(C_{\sigma}\)-uniqueness. We show that a dcpo \(D\) is \(C_{\sigma}\)-unique if and only if \(D\in\mathcal R_2\) and \(D\) contains no proper sub-dcpo \(P\) satisfying
\[
  \cl_{\sigma(D)}(P\cap\down x)=\down x
  \qquad \text{and} \qquad
  \cl_{\sigma(D)}(A)\cap P=\cl_{\sigma(P)}(A)
\]
for every \(x\in D\) and \(A\subseteq P\).
The condition \(D\in\mathcal R_2\) comes from the existing \(2\)-reflection theory, while the new contribution here is the Scott-closure characterization of those sub-dcpos that preserve the lattice of Scott-closed sets. Then, building on this characterization and the accompanying reconstruction tools, we answer three open problems mentioned above.

The main results are as follows.
\begin{enumerate}[leftmargin=2.2em]
\item A dcpo $D$ is $C_{\sigma}$-unique exactly when $D\in\Rtwo$ and no proper sub-dcpo satisfies both closure conditions. The characterization is intrinsic in that, conditional on $D\in\Rtwo$, it tests sub-dcpos of $D$ rather than arbitrary comparison dcpos.
\item Local domination and directed generation by $C_{\sigma}$-unique elements suffice for reconstruction. An arbitrary-product criterion then proves that $\J^I$ is $C_{\sigma}$-unique for every set $I$.
\item There is a countable sober $C_{\sigma}$-unique dcpo $K$ such that $K\times(0,1]$ is sober but not $C_{\sigma}$-unique. The witnessing sub-dcpo is countable.
\item There are a sober countable frame $F$ and a countable dcpo $\Pframe$ with $\Gamma(F)\cong\Gamma(\Pframe)$ but $F\not\cong\Pframe$.
\end{enumerate}

Section~2 recalls the required preliminaries. Section~3 gives the closure criterion and the intrinsic characterization. Section~4 develops the independent local reconstruction argument and applies it to powers of Johnstone's dcpo. Sections~5 and~6 give the counterexamples. Section~7 summarizes the answers to the cited questions and states two further problems.

\section{Preliminaries}\label{sec:preliminaries}

Standard terminology follows \cite{gierz2003,goubault2013}. 
For a poset $P$ and $A\subseteq P$, write
\(
  \down A=\{x:\exists a\in A,\ x\leq a\} \) and \(
  \up A=\{x:\exists a\in A,\ a\leq x\}.
\)
All directed subsets are supposed to be nonempty. We use
\(
  \mathbb N=\{0,1,2,\ldots\},\
  \mathbb N_{>0}=\{1,2,\ldots\}.
\)
A poset is a dcpo if every directed subset has a supremum. Its Scott topology is denoted by $\sigma(P)$, its Scott space by $\Sigma P$, and its complete lattice of Scott-closed subsets by $\Gamma(P)$. A map between dcpos is Scott continuous if and only if it is monotone and preserves directed suprema. We write $\mathbf{DCPO}$ for the category of dcpos and Scott-continuous maps, and $[P\to Q]$ for the pointwise ordered dcpo of Scott-continuous maps when $Q$ is a dcpo.

A nonempty subset $A$ of a topological space is \emph{irreducible} if
$A\subseteq F\cup G$, with $F$ and $G$ closed, implies
$A\subseteq F$ or $A\subseteq G$. A dcpo is \emph{sober} if every
nonempty irreducible Scott-closed subset is principal. It is \emph{bounded sober} if this holds for every such subset having an upper bound; equivalently, every principal ideal is sober \cite[Section~3]{XuZhao2019}.
We write
\[
  \Irr(\Gamma(D))=
  \{A\in\Gamma(D):A\neq\varnothing\text{ and }A
       \text{ is irreducible in }\Sigma D\}.
\]
Equivalently, these are the nonzero finite join-prime elements of the lattice $\Gamma(D)$. This notation does not refer to irreducible closed subsets of the Scott space of $\Gamma(D)$.

A full subcategory $\mathcal C$ of $\DCPO$ is \emph{$\Gamma$-faithful} if
\[
  D,Q\in\mathcal C,\quad \Gamma(D)\cong\Gamma(Q)
       \quad\Longrightarrow\quad D\cong Q.
\]
This is an objectwise property, not faithfulness on morphisms. In contrast, $C_{\sigma}$-uniqueness of an object $D$ compares it with every dcpo $Q$.

A subset $E$ of a dcpo $D$ is a \emph{sub-dcpo} if it is closed under directed suprema computed in $D$.
For a poset $D$ and finite subsets $F,G\subseteq D$, write $F\leq_{\mathrm{Sm}}G$ if
$\up G\subseteq\up F$; this is the Smyth preorder
\cite[Definition~III-3.2]{gierz2003}.

Let $D$ be a dcpo. A finite subset $F\subseteq D$ is \emph{way below} $x\in D$, written $F\ll x$, if for every directed $E\subseteq D$,
\(
  x\leq\bigvee E\) implies \( E\cap\up F\neq\varnothing.
\)
Put $\operatorname{fin}_D(x)=\{F\subseteq D:F\text{ is finite and }F\ll x\}$. The dcpo $D$ is \emph{quasicontinuous} if $\operatorname{fin}_D(x)$ is directed in the Smyth preorder and
\(
  \up x=\bigcap_{F\in\operatorname{fin}_D(x)}\up F
\)
for every $x\in D$. If $D$ is quasicontinuous, $U\in\sigma(D)$ and $x\in U$, then some finite $F\ll x$ satisfies $\up F\subseteq U$ \cite[Proposition~III-3.6]{gierz2003}.
A dcpo is \emph{locally quasicontinuous} if every principal ideal, with the inherited order, is quasicontinuous.
Every quasicontinuous dcpo is sober \cite[Proposition~III-3.7]{gierz2003} and $C_{\sigma}$-unique \cite[Lemma~3.8]{ZhaoXu2018}. Finite products preserve quasicontinuity and local quasicontinuity \cite[Remark~3.11(i)]{XuZhao2019}.

We begin with several basic lemmas that will be used throughout the paper.

\begin{lemma}\label{lem:closed-subdcpo}
Let $D$ be a dcpo and let $C\in\Gamma(D)$ be nonempty. Then
\[
  \Gamma(C)=\{F\in\Gamma(D):F\subseteq C\}
            =\down_{\Gamma(D)}C.
\]
Moreover,
\[
  \Irr(\Gamma(C))=\{A\in\Irr(\Gamma(D)):A\subseteq C\}
          =\Irr(\Gamma(D))\cap\down_{\Gamma(D)}C.
\]
\end{lemma}

\begin{proof}
A subset $F\subseteq C$ is lower in $C$ if and only if it is lower in $D$,
because $C$ is lower. Directed subsets of $F$ have the same suprema in $C$
and $D$. This proves the first equality. The second follows because
irreducibility is tested against finite unions of Scott-closed sets, whose
restrictions to $C$ are precisely the Scott-closed sets of $C$.
\end{proof}

\begin{lemma}\label{lem:principal-directed-join}
Let $A$ be a directed subset of a dcpo $D$. Then
\[
  \bigvee_{\Gamma(D)}\{\down_Da:a\in A\}=\down_D\bigvee A.
\]
\end{lemma}

\begin{proof}
The right-hand side is a Scott-closed upper bound of the family. Conversely,
any Scott-closed set containing every $\down_Da$ contains $A$, hence contains
$\bigvee A$ and therefore $\down_D\bigvee A$.
\end{proof}

Every directed subset of a countable poset has an increasing cofinal
sequence. Indeed, enumerate it as $d_0,d_1,\ldots$, allowing repetitions if
it is finite, and choose recursively $c_0=d_0$ and $c_{n+1}$ above both
$c_n$ and $d_{n+1}$. Consequently, a lower subset of a countable dcpo is
Scott closed if and only if it is closed under suprema of increasing
sequences.

\subsection{Skula topology and Skula-density}\label{subsec:skula-topology}

Let $X$ be a $T_0$ topological space with topology $\tau$.
The \emph{Skula topology} of $X$, denoted by $\tau_S$, is the smallest
topology containing both the open subsets and the closed subsets of
$(X,\tau)$ \cite{Skula1969}. For $A\subseteq X$, write
$\cl_{\tau_S}(A)$ for the closure of $A$ in $(X,\tau_S)$. The subset $A$
is called \emph{Skula-dense} in $X$ if $\cl_{\tau_S}(A)=X$.
For a dcpo $D$ and a subset $A\subseteq D$, we write
$\cl_{\sigma(D)}(A)$ for the closure of $A$ in the Scott topology
$\sigma(D)$. Equivalently, $\cl_{\sigma(D)}(A)$ is the least lower
subset of $D$ containing $A$ and closed under suprema of nonempty directed
subsets.

The pointwise criterion for Skula-density has the following form for Scott spaces.

\begin{lemma}\label{lem:Skula-density-equivalence}
Let $D$ be a dcpo and let $P\subseteq D$ be any subset. Then $P$ is Skula-dense in $\Sigma D$ if and only if
\[
  \cl_{\sigma(D)}(P\cap\down_Dx)=\down_Dx
  \qquad(x\in D).
\]
\end{lemma}

\begin{proof}
A basic Skula neighbourhood of $x$ is $O\cap C$, where $O$ is Scott open and $C$ is Scott closed, with $x\in O\cap C$. Since $\down_Dx\subseteq C$, every such neighbourhood contains $O\cap\down_Dx$. Conversely, $O\cap\down_Dx$ is itself Skula open. Thus $P$ is Skula-dense exactly when every Scott-open neighbourhood $O$ of every $x$ meets $P\cap\down_Dx$, that is, when $x\in\cl_{\sigma(D)}(P\cap\down_Dx)$. This closure is contained in $\down_Dx$ and is lower, so the latter condition is equivalent to the displayed equality.
\end{proof}

\subsection{The two-element reflection}\label{subsec:reflection}

Equability was introduced in the study of observationally induced algebras
\cite{bks2014}. For a class $\mathcal E$ of morphisms
in $\DCPO$, let $\mathcal E^{\perp}$ denote the full subcategory of
all dcpos $C$ such that, for every $e:A\to B$ in $\mathcal E$,
precomposition induces an isomorphism
\[
  e^*:[B\to C]\longrightarrow[A\to C].
\]
Here $[A\to C]$ denotes the dcpo of Scott-continuous maps $A\to C$,
ordered pointwise; directed suprema are computed pointwise.
Let $\two=\{0<1\}$. For a Scott-continuous map $e:A\to B$,
precomposition gives
\[
  e^*:[B\to\two]\longrightarrow[A\to\two],\qquad f\longmapsto f\circ e.
\]
\begin{definition}\label{def:E2}
A morphism $e:A\to B$ belongs to $E_{\two}$ if $e^*$ is an isomorphism of
dcpos. Such a morphism is called \emph{$2$-equable}. Equivalently,
\(
  e^{-1}:\sigma(B)\longrightarrow\sigma(A)
\)
is an order isomorphism. The full subcategory
$\Rtwo=E_{\two}^{\perp}$ is the reflective hull of $\two$ in $\DCPO$ \cite[Theorem~2.7]{HeLyuChenKou2026}. We write
\(
  \lambda_D:D\longrightarrow L_2 D
\)
for the $\Rtwo$-reflection unit.
\end{definition}

Taking complements shows that $e\in E_{\two}$ if and only if inverse image
induces an order isomorphism between the Scott-closed-set lattices.

\begin{theorem}[{\cite[Theorems~2.5 and~3.5, Lemma~5.1,
Theorem~5.5, and Proposition~6.13]{HeLyuChenKou2026}}]
\label{thm:reflection-package}
For every dcpo $D$, the following hold.
\begin{enumerate}[label=\textup{(\roman*)},leftmargin=2.2em]
\item The reflection unit $\lambda_D$ belongs to $E_{\two}$ and induces order
      isomorphisms
      \[
        \sigma(L_2 D)\cong\sigma(D),\qquad
        \Gamma(L_2 D)\cong\Gamma(D).
      \]
\item If $D\in\Rtwo$ and $e:D\to E$ belongs to $E_{\two}$, then $e$ is an
      isomorphism in $\DCPO$.
\item The category $\Rtwo$ is $\Gamma$-faithful.
\item Every $C_{\sigma}$-unique dcpo belongs to $\Rtwo$.
\item Every sober dcpo belongs to $\Rtwo$
\end{enumerate}
\end{theorem}

\begin{lemma}\label{lem:E2-image}
The class $E_{\two}$ contains all isomorphisms and is closed under composition.
If $e:A\to B$ belongs to $E_{\two}$, then $e$ is an order embedding, its image
$e[A]$ is a sub-dcpo of $B$, and the inclusion $j:e[A]\hookrightarrow B$
belongs to $E_{\two}$.
\end{lemma}

\begin{proof}
The first two assertions follow from
$(g\circ e)^{-1}=e^{-1}\circ g^{-1}$ on Scott-open sets. By
\cite[Lemma~3.4]{HeLyuChenKou2026}, surjectivity of
$e^{-1}:\sigma(B)\to\sigma(A)$ implies that $e$ is an order embedding and
that $e[A]$ is a sub-dcpo of $B$. Let $\bar e:A\to e[A]$ be the
corestriction. It is a dcpo isomorphism and $e=j\circ\bar e$. On open sets,
$e^{-1}=\bar e^{-1}\circ j^{-1}$. Both $e^{-1}$ and $\bar e^{-1}$ are
order isomorphisms, hence so is $j^{-1}$.
\end{proof}

For a dcpo $D$, let
\(
  \widehat D=\Irr(\Gamma(D))
\)
be the dcpo of nonempty irreducible Scott-closed subsets of $D$, ordered by
inclusion. We identify $D$ with its principal image under
\[
  \eta_D:D\longrightarrow\widehat D,\qquad x\longmapsto\down_Dx.
\]
The dcpo $\widehat D$ is the order sobrification of $D$.

\section{An intrinsic characterization of \texorpdfstring{$C_{\sigma}$}{C-sigma}-unique dcpos}\label{sec:characterization}

In this section, we derive an intrinsic characterization of \(C_{\sigma}\)-uniqueness.
We first identify the sub-dcpos that preserve Scott-closed-set lattices and give two sufficient conditions for constructing them. The two-element reflection then converts this criterion into a characterization of $C_{\sigma}$-uniqueness.

\subsection{Skula-density and closure compatibility}\label{subsec:trace-cores}

\begin{definition}\label{def:E2-subdcpo}
A sub-dcpo $P\subseteq D$ is \emph{$2$-equable} if its inclusion belongs to $E_{\two}$. Equivalently, restriction is an order isomorphism
\[
  \rho:\Gamma(D)\longrightarrow\Gamma(P),\qquad C\longmapsto C\cap P.
\]
\end{definition}

\begin{theorem}\label{thm:closure-characterization}
Let $P\subseteq D$ be a sub-dcpo. Then $P$ is $2$-equable if and only if
\begin{align}
  \cl_{\sigma(D)}(P\cap\down_Dx)&=\down_Dx
     &&\text{for every }x\in D, \tag{SD}\label{eq:SD}\\
  \cl_{\sigma(D)}(A)\cap P&=\cl_{\sigma(P)}(A)
     &&\text{for every }A\subseteq P. \tag{SC}\label{eq:SC}
\end{align}
When these conditions hold, restriction has inverse $C\mapsto\cl_{\sigma(D)}(C)$.
\end{theorem}

\begin{proof}
Write $\rho(C)=C\cap P$ and first suppose that $\rho$ is an order isomorphism. For $x\in D$, set $H=\cl_{\sigma(D)}(P\cap\down_Dx)$. Then $H\subseteq\down_Dx$ and $H\cap P=(\down_Dx)\cap P$. Injectivity gives $H=\down_Dx$, proving \eqref{eq:SD}.

For $A\subseteq P$, continuity of the inclusion gives
\[
  \cl_{\sigma(P)}(A)\subseteq\cl_{\sigma(D)}(A)\cap P.
\]
Choose $G\in\Gamma(D)$ with $G\cap P=\cl_{\sigma(P)}(A)$. Since $A\subseteq G$, one has $\cl_{\sigma(D)}(A)\subseteq G$, which proves the reverse inclusion and hence \eqref{eq:SC}.

Conversely, assume \eqref{eq:SD} and \eqref{eq:SC}. If $G,H\in\Gamma(D)$ have the same trace on $P$ and $x\in G$, then $P\cap\down_Dx\subseteq H$. By \eqref{eq:SD}, $x\in H$. Interchanging $G$ and $H$ proves injectivity of $\rho$. If $C\in\Gamma(P)$, then \eqref{eq:SC} gives $\cl_{\sigma(D)}(C)\cap P=C$. Thus $\rho$ is bijective with the stated monotone inverse, and hence is an order isomorphism.
\end{proof}

Condition \eqref{eq:SD} is Skula-density by \cref{lem:Skula-density-equivalence}. Condition \eqref{eq:SC} says precisely that $\Sigma P$ carries the subspace topology inherited from $\Sigma D$, since a topology is determined by its closure operator. Thus a sub-dcpo is $2$-equable exactly when its Scott space is a Skula-dense topological subspace of the ambient Scott space.

\subsection{Sufficient conditions for the closure identities}\label{subsec:closure-sufficient}

The following two conditions will be used in the counterexamples.

\begin{definition}\label{def:directed-lifting-bridge-density}
Let $P\subseteq D$ be a sub-dcpo. It has \emph{directed lifting} if, whenever $E\subseteq D$ is directed and $p\in P$ satisfies $p\leq\bigvee E$, there is a directed set $B\subseteq P$ with $\bigvee B=p$ such that each $b\in B$ lies below some member of $E$.
It has \emph{bridge density} if, for each $x\in D$, there is a directed set $V_x\subseteq D$ with $\bigvee V_x=x$ such that every $v\in V_x$ admits $p_v\in P$ with $v\leq p_v\leq x$.
\end{definition}

\begin{proposition}\label{prop:directed-lifting-bridge-density}
Bridge density implies \eqref{eq:SD}, and directed lifting implies \eqref{eq:SC}. Hence the two conditions together make $P\hookrightarrow D$ $2$-equable.
\end{proposition}

\begin{proof}
Under bridge density, every $v\in V_x$ lies below a point of $P\cap\down_Dx$. Consequently, its Scott closure contains $V_x$, then $x=\bigvee V_x$, and hence all of $\down_Dx$. This proves \eqref{eq:SD}.

Assume directed lifting. For $C\in\Gamma(P)$ put
\[
  H_C=\{x\in D:P\cap\down_Dx\subseteq C\}.
\]
This is a lower set. Let $E\subseteq H_C$ be directed and let $p\in P$ satisfy $p\leq\bigvee E$. Choose a lifting set $B\subseteq P$ with supremum $p$. For every $b\in B$, some $e\in E$ satisfies $b\leq e$; since $e\in H_C$, one has $b\in C$. Thus $B\subseteq C$, and Scott closedness of $C$ gives $p\in C$. Therefore $\bigvee E\in H_C$, so $H_C\in\Gamma(D)$.

Since $C$ is lower in $P$, one has $H_C\cap P=C$. For $A\subseteq P$, take $C=\cl_{\sigma(P)}(A)$. The inclusion $A\subseteq H_C$ gives
\[
  \cl_{\sigma(D)}(A)\cap P\subseteq H_C\cap P=C.
\]
The reverse inclusion follows from continuity of $P\hookrightarrow D$. This proves \eqref{eq:SC}.
\end{proof}

\begin{remark}
% Directed lifting is sufficient but not necessary for \eqref{eq:SC}. Let
% \[
%   D=\{\bot,p,\top\}\cup\{e_n:n\in\N\},
% \]
% ordered by $\bot<e_0<e_1<\cdots<\top$ and $\bot<p<\top$, with $p$ incomparable with all $e_n$. For $P=D$, both closure identities hold. However, $E=\{e_n:n\in\N\}$ has supremum $\top$, and the only point below both $p$ and some $e_n$ is $\bot$. No lifting set for $p$ can therefore have supremum $p$.
The two implications above are only sufficient conditions and are not reversible in general. For instance, the initial sub-dcpo \(P^\times_0\subseteq K\times\Ichain\) constructed in Section~\ref{sec:product-counterexample} satisfies \textup{(SD)} but not bridge density. On the other hand, if
\(
D=\{\bot,p,\top\}\cup\{e_n:n\in\mathbb N\},
\)
with \(\bot<e_0<e_1<\cdots<\top\) and \(\bot<p<\top\), where \(p\) is incomparable with every \(e_n\), then taking \(P=D\) gives \textup{(SC)} while directed lifting fails.
\end{remark}

\subsection{Reflection and intrinsic rigidity}\label{subsec:canonical-envelope}

The reflection facts from \cref{thm:reflection-package} give the following reduction.

\begin{theorem}\label{thm:canonical-envelope}
For dcpos $D$ and $Q$, one has $\Gamma(Q)\cong\Gamma(D)$ if and only if $Q$ is isomorphic to a $2$-equable sub-dcpo of $L_2D$.
\end{theorem}

\begin{proof}
If $\Gamma(Q)\cong\Gamma(D)$, then
\[
  \Gamma(L_2Q)\cong\Gamma(Q)\cong\Gamma(D)\cong\Gamma(L_2D).
\]
By $\Gamma$-faithfulness of $\Rtwo$, there is an isomorphism $h:L_2Q\to L_2D$. The composite $h\circ\lambda_Q$ belongs to $E_{\two}$. Its image is the required sub-dcpo by \cref{lem:E2-image}.

Conversely, if $Q$ is isomorphic to a $2$-equable sub-dcpo $P\subseteq L_2D$, then $\Gamma(Q)\cong\Gamma(P)\cong\Gamma(L_2D)\cong\Gamma(D)$.
\end{proof}

\begin{theorem}\label{thm:intrinsic-rigidity}
For a dcpo $D$, the following are equivalent.
\begin{enumerate}[label=\textup{(\roman*)},leftmargin=2.2em]
\item $D$ is $C_{\sigma}$-unique.
\item The dcpo $L_2D$ has no proper $2$-equable sub-dcpo.
\item $D\in\Rtwo$, and every sub-dcpo $P\subseteq D$ satisfying \eqref{eq:SD} and \eqref{eq:SC} equals $D$.
\end{enumerate}
Equivalently, the sub-dcpos in \textup{(ii)} may be specified by \eqref{eq:SD} and \eqref{eq:SC} in the ambient dcpo $L_2D$.
\end{theorem}

\begin{proof}
Suppose that $D$ is $C_{\sigma}$-unique. Then $D\in\Rtwo$ by \cref{thm:reflection-package}(iv). If $P\subseteq L_2D$ is $2$-equable, then $\Gamma(P)\cong\Gamma(D)$, so $P\cong D$ and $P\in\Rtwo$. By \cref{thm:reflection-package}(ii), the inclusion $P\hookrightarrow L_2D$ is itself an isomorphism, hence is surjective. Thus $P=L_2D$.

Assume \textup{(ii)}. By \cref{lem:E2-image}, the image of $\lambda_D$ is a $2$-equable sub-dcpo of $L_2D$. It must be all of $L_2D$, so $\lambda_D$ is an isomorphism and $D\in\Rtwo$. Identifying $D$ with $L_2D$ and applying \cref{thm:closure-characterization} proves \textup{(iii)}.

Finally, suppose \textup{(iii)} and $\Gamma(Q)\cong\Gamma(D)$. By \cref{thm:canonical-envelope}, $Q$ is isomorphic to a $2$-equable sub-dcpo $P\subseteq D$. The closure criterion and \textup{(iii)} give $P=D$, and hence $Q\cong D$.
\end{proof}

For later use, we record the corresponding relation with sobrification. Here the topological sobrification has underlying set $\widehat P$ and hull--kernel closed sets
\[
  \widehat C=\{H\in\widehat P:H\subseteq C\},\qquad C\in\Gamma(P),
\]
as in \cite[Chapter~8]{goubault2013}; its topology is not assumed to be the Scott topology of $\widehat P$.

\begin{corollary}\label{cor:sober-equable-completions}
Let $D$ be sober and let $i:P\hookrightarrow D$ be a $2$-equable sub-dcpo. Then $i$ is both the sobrification map of $\Sigma P$ and the $\Rtwo$-reflection of $P$. More precisely,
\[
  \Psi:D\longrightarrow\widehat P,\qquad
  \Psi(x)=P\cap\down_Dx,
\]
is an order isomorphism and a homeomorphism from $\Sigma D$ to the topological sobrification of $\Sigma P$. Consequently,
\[
  D\cong L_2P\cong\widehat P.
\]
If $P\subsetneq D$, then $P\notin\Rtwo$ and $\Sigma P$ is not sober.
\end{corollary}

\begin{proof}
Restriction $\rho:\Gamma(D)\to\Gamma(P)$ preserves irreducible elements. Since $D$ is sober, all members of $\Irr(\Gamma(D))$ are principal ideals. Hence $\Psi$ is an order isomorphism onto $\widehat P$.

For $C\in\Gamma(P)$, put $G=\rho^{-1}(C)$. Then
\[
  \Psi(x)\subseteq C
  \quad\Longleftrightarrow\quad
  \rho(\down_Dx)\subseteq\rho(G)
  \quad\Longleftrightarrow\quad x\in G.
\]
Thus $\Psi^{-1}(\widehat C)=G$, and $G$ ranges over all of $\Gamma(D)$. This proves that $\Psi$ is a homeomorphism. For $p\in P$, one has $\Psi(i(p))=\down_Pp$, so $i$ corresponds to the canonical sobrification map.

Moreover, $D\in\Rtwo=E_{\two}^{\perp}$. Since $i\in E_{\two}$, precomposition induces $[D\to R]\cong[P\to R]$ for every $R\in\Rtwo$. This is the universal property of the $\Rtwo$-reflection, so $D\cong L_2P$. If $P\in\Rtwo$, \cref{thm:reflection-package}(ii) would make $i$ an isomorphism. Hence a proper such $P$ lies outside $\Rtwo$ and, in particular, is not sober.
\end{proof}

\section{Local reconstruction and powers of Johnstone's dcpo}\label{sec:reconstruction-tools}

This section develops a local reconstruction theorem and applies it to the
Johnstone-power problem. We first replace bounded sobriety by local domination
in the generation theorem of Xu and Zhao, then derive an arbitrary-product
theorem and prove that $\J^I$ is $C_{\sigma}$-unique for every set $I$.

\subsection{A reconstruction theorem under local domination}\label{subsec:local-generation}

Following Ho, Goubault-Larrecq, Jung and Xi
\cite[Definition~2.1]{ho2018}, a dcpo $S$ is
\emph{strongly complete} if every nonempty irreducible subset of its Scott
space has a supremum. For $u,v\in S$, where $S$ is strongly complete, write $u\prec_Sv$ if, for every Scott-closed irreducible subset
$\mathcal A\subseteq S$,
\(
  v\leq\bigvee\mathcal A \) implies \( u\in\mathcal A.
\)
An element $u$ is \emph{$\prec$-compact} if $u\prec_Su$.

For $A_0,A\in\widehat D$, write $A_0\triangleleft A$ if there exists
$x\in A$ such that $A_0\subseteq\down_Dx$. The dcpo $D$ is
\emph{dominated} if, for every $A\in\widehat D$, the set
\[
  \nabla A=\{A_0\in\widehat D:A_0\triangleleft A\}
\]
is Scott closed in $\widehat D$. It is \emph{locally dominated} if every
principal ideal is dominated \cite[Definition~2.4]{XuZhao2019}.

\begin{proposition}[{\cite[Propositions~2.2(4), 2.9, 2.12 and
Theorem~2.15]{ho2018}}]\label{prop:dominated-facts}
Let $D$ be a dcpo.
\begin{enumerate}[label=\textup{(\roman*)},leftmargin=2.2em]
\item $\widehat D$ is strongly complete.
\item Every principal ideal $\down_Dx$, regarded as an element of
      $\widehat D$, is $\prec$-compact.
\item If $D$ is dominated, then the $\prec$-compact elements of
      $\widehat D$ are exactly the principal ideals $\down_Dx$, $x\in D$.
\item Every strongly complete dcpo is dominated. In particular, every sober
      dcpo is dominated.
\end{enumerate}
\end{proposition}

Zhao and Xu's generation theorem assumes bounded sobriety and
quasicontinuous generators \cite[Theorem~3.9]{ZhaoXu2018}. Xu and Zhao later
replaced the latter by $C_{\sigma}$-unique generators
\cite[Proposition~3.2]{XuZhao2019}. We retain this more general generator
condition and weaken only the ambient hypothesis, from bounded sobriety to
local domination.

\begin{definition}\label{def:directed-generation}
Let $D$ be a dcpo.
\begin{enumerate}[label=\textup{(\roman*)},leftmargin=2.2em]
\item An element $b\in D$ is a \emph{$C_{\sigma}$-unique element} if
      the dcpo $\down_Db$ is $C_{\sigma}$-unique.
\item The dcpo $D$ is \emph{directedly $C_{\sigma}$-generated} if, for every
      $x\in D$, there exists a directed set $B_x\subseteq\down_Dx$
      consisting of $C_{\sigma}$-unique elements such that
      $x=\bigvee B_x$.
\end{enumerate}
\end{definition}

\begin{theorem}\label{thm:local-reconstruction}
Let $D$ be a locally dominated and directedly $C_{\sigma}$-generated dcpo.
Then $D$ is $C_{\sigma}$-unique.
\end{theorem}

\begin{proof}
Let $Q$ be a dcpo and let
\(\Phi:\Gamma(D)\longrightarrow\Gamma(Q)\)
be an order isomorphism. If $D=\varnothing$, then $\Gamma(D)$ has one element, so $Q=\varnothing$ and the result is immediate. Assume $D\neq\varnothing$. We construct an order isomorphism $f:D\to Q$ in
three steps.

\emph{Step 1: reconstruction from principal ideals.}
Let $b$ be a $C_{\sigma}$-unique element of $D$. By
\cref{lem:closed-subdcpo}, the restriction of $\Phi$ yields
\[
  \Gamma(\down_Db)
  \cong\down_{\Gamma(D)}(\down_Db)
  \cong\down_{\Gamma(Q)}\Phi(\down_Db)
  \cong\Gamma\bigl(\Phi(\down_Db)\bigr).
\]
The Scott-closed set $\Phi(\down_Db)$ is nonempty and hence is a dcpo with
the inherited order. Since $\down_Db$ is $C_{\sigma}$-unique,
$\Phi(\down_Db)$ has a largest element, denoted by $f(b)$, and
\begin{equation}\label{eq:image-generator}
  \Phi(\down_Db)=\down_Qf(b).
\end{equation}
Consequently, for $C_{\sigma}$-unique elements $b,c$,
\begin{equation}\label{eq:generator-order}
  b\leq c\quad\Longleftrightarrow\quad f(b)\leq f(c).
\end{equation}

For $x\in D$, choose a directed set $B_x$ of $C_{\sigma}$-unique elements
with $x=\bigvee B_x$. By \eqref{eq:generator-order}, $f(B_x)$ is directed,
and \cref{lem:principal-directed-join} gives
\[
\begin{aligned}
  \Phi(\down_Dx)
  &=\Phi\left(\bigvee_{\Gamma(D)}\{\down_Db:b\in B_x\}\right)\\
  &=\bigvee_{\Gamma(Q)}\{\down_Qf(b):b\in B_x\}\\
  &=\down_Q\left(\bigvee_{b\in B_x}f(b)\right).
\end{aligned}
\]
Thus $\Phi(\down_Dx)$ has a unique largest element. Define $f(x)$ by
\begin{equation}\label{eq:image-principal}
  \Phi(\down_Dx)=\down_Qf(x).
\end{equation}
This definition is independent of the choice of $B_x$, and it yields
\begin{equation}\label{eq:order-embedding}
  x\leq x'\quad\Longleftrightarrow\quad f(x)\leq f(x').
\end{equation}
Hence $f$ is an order embedding.

\emph{Step 2: preservation of directed suprema.}
If $A\subseteq D$ is directed, then
\cref{lem:principal-directed-join} and \eqref{eq:image-principal} give
\[
\begin{aligned}
  \down_Qf\!\left(\bigvee A\right)
  &=\Phi\left(\down_D\bigvee A\right)\\
  &=\Phi\left(\bigvee_{\Gamma(D)}\{\down_Da:a\in A\}\right)\\
  &=\bigvee_{\Gamma(Q)}\{\down_Qf(a):a\in A\}\\
  &=\down_Q\bigvee f(A).
\end{aligned}
\]
Therefore
\begin{equation}\label{eq:preserves-directed-sups}
  f\!\left(\bigvee A\right)=\bigvee f(A).
\end{equation}

\emph{Step 3: local domination forces the image to be lower.}
Let $y\leq f(x)$ and put $C=\Phi^{-1}(\down_Qy)$, then $C\in\Irr(\Gamma(D))$. Moreover, $\down_Qy\subseteq\Phi(\down_Dx)$, so
$C\subseteq\down_Dx$. By \cref{lem:closed-subdcpo}, $\Phi$ therefore
restricts to an order isomorphism
\[
  \Gamma(\down_Dx)\longrightarrow\Gamma(\down_Qf(x))
\]
and hence to an order isomorphism
\[
  \widehat{\down_Dx}\longrightarrow\widehat{\down_Qf(x)}
\]
that sends $C$ to $\down_Qy$. The relation $\prec$ is invariant under order
isomorphisms of strongly complete dcpos, since such isomorphisms preserve
Scott-closed irreducible subsets and their suprema. By
\cref{prop:dominated-facts}(i),(ii), $\down_Qy$ is $\prec$-compact in
$\widehat{\down_Qf(x)}$, and hence $C$ is $\prec$-compact in
$\widehat{\down_Dx}$. Since $D$ is locally dominated, $\down_Dx$ is
dominated. Thus \cref{prop:dominated-facts}(iii) gives some $x'\leq x$ with
$C=\down_Dx'$. Equation~\eqref{eq:image-principal} now yields
$\down_Qy=\down_Qf(x')$, and hence $y=f(x')$. Therefore $f(D)$ is lower in
$Q$.

By \eqref{eq:preserves-directed-sups}, $f(D)$ is also closed under directed
suprema in $Q$, and hence is Scott closed. Finally,
\[
  Q=\Phi(D)
   =\bigvee_{\Gamma(Q)}\{\Phi(\down_Dx):x\in D\}
   =\bigvee_{\Gamma(Q)}\{\down_Qf(x):x\in D\}.
\]
The Scott-closed set $f(D)$ contains every $\down_Qf(x)$, so it contains
their join. Hence $f(D)=Q$, and $f$ is an order isomorphism.
\end{proof}

\begin{remark}\label{rem:comparison-xuzhao}
The construction of $f$ from principal ideals and the directed-generation
argument are inherited from \cite[Theorem~3.9]{ZhaoXu2018} and
\cite[Proposition~3.2]{XuZhao2019}. The new step is the proof that $f(D)$ is
lower: local domination principalizes $C$ inside
$\widehat{\down_Dx}$ and thereby replaces the use of bounded sobriety.
Since every bounded sober dcpo is locally dominated, the 2019 result is a
special case. Local domination is strictly weaker in general
\cite[Remark~3.8(2) and Proposition~4.10]{XuZhao2019}, although no strict
separation is claimed within the smaller class of directedly
$C_{\sigma}$-generated dcpos.
\end{remark}

\subsection{An arbitrary-product theorem}\label{subsec:arbitrary-products}

The preceding theorem applies to arbitrary products once its two hypotheses
are verified. Strong completeness supplies local domination of principal
ideals, while finite-coordinate truncations supply the
$C_{\sigma}$-unique generators.

\begin{proposition}\label{prop:product-strongly-complete}
An arbitrary product of strongly complete dcpos is strongly complete.
\end{proposition}

\begin{proof}
Let $(S_i)_{i\in I}$ be strongly complete dcpos and put
$S=\prod_{i\in I}S_i$. Let $A\subseteq S$ be a nonempty irreducible subset.
Each coordinate projection $\pi_i:S\to S_i$ is Scott continuous, so
$\pi_i(A)$ is nonempty and irreducible. Put $s_i=\bigvee\pi_i(A)$. Then
$s=(s_i)_{i\in I}$ is an upper bound of $A$, and every upper bound $u$ of
$A$ satisfies $s_i\leq u_i$ for all $i$. Hence $s=\bigvee A$.
\end{proof}

\begin{theorem}\label{thm:arbitrary-product}
Let $(D_i)_{i\in I}$ be a family of locally quasicontinuous dcpos such that
every point of each $D_i$ lies above a minimal element. Then
$\prod_{i\in I}D_i$ is $C_{\sigma}$-unique.
\end{theorem}

\begin{proof}
The empty product is a one-element dcpo and is $C_{\sigma}$-unique. If a factor is empty, so is the product, and the conclusion is again immediate. Otherwise put $D=\prod_{i\in I}D_i$ and fix $x=(x_i)_{i\in I}\in D$. Its principal
ideal is order-isomorphic to
\begin{equation}\label{eq:principal-product}
  \down_Dx\cong\prod_{i\in I}\down_{D_i}x_i.
\end{equation}
Each factor on the right is quasicontinuous, hence sober and strongly
complete. By \cref{prop:product-strongly-complete}, the product in
\eqref{eq:principal-product} is strongly complete and therefore dominated
by \cref{prop:dominated-facts}(iv). Thus $D$ is locally dominated.

For each $i\in I$, choose a minimal element $r_i\leq x_i$. For a finite
subset $H\subseteq I$, define $x^H\in D$ by
\begin{equation}\label{eq:finite-coordinate-truncation}
  x^H_i=
  \begin{cases}
    x_i,&i\in H,\\
    r_i,&i\notin H.
  \end{cases}
\end{equation}
The family $\{x^H:H\subseteq I\text{ finite}\}$ is directed and has
supremum $x$. Since $\down_{D_i}r_i=\{r_i\}$,
\begin{equation}\label{eq:truncation-principal}
  \down_Dx^H\cong\prod_{i\in H}\down_{D_i}x_i.
\end{equation}
The finite product in \eqref{eq:truncation-principal} is quasicontinuous
\cite[Remark~3.11(i)]{XuZhao2019} and hence $C_{\sigma}$-unique
\cite[Lemma~3.8]{ZhaoXu2018}. Therefore every $x^H$ is a
$C_{\sigma}$-unique element, and $D$ is directedly
$C_{\sigma}$-generated. Apply \cref{thm:local-reconstruction}.
\end{proof}

\subsection{Powers of Johnstone's dcpo}\label{sec:johnstone}

The countable-power problem asks whether $\J^{\N}$ is $C_{\sigma}$-unique
\cite[Problem~9]{Zhao2022Survey}; it was repeated as
\cite[Problem~1]{Zhao2025Johnstone}. We prove the stronger arbitrary-index
statement.

Let $\infty$ be larger than every natural number. Johnstone's dcpo is
\[
  \J=\N\times(\N\cup\{\infty\}),
\]
ordered by
\begin{equation}\label{eq:johnstone-order}
  (m,n)\leq(m',n')
  \quad\Longleftrightarrow\quad
  \bigl(m=m'\text{ and }n\leq n'\bigr)
  \ \text{or}\
  \bigl(n'=\infty\text{ and }n\leq m'\bigr).
\end{equation}
This is the classical non-sober dcpo introduced by Johnstone
\cite{Johnstone1981}. Its maximal elements are the points $(m,\infty)$.

Every principal ideal of $\J$ is quasicontinuous
\cite[Proposition~3.5]{XuZhao2019}, so $\J$ is locally
quasicontinuous. Moreover, every $(m,n)\in\J$ lies above the minimal element
$(m,0)$. Indeed, if $(p,q)\leq(m,0)$, the second clause of
\eqref{eq:johnstone-order} is unavailable, so $p=m$ and $q=0$.

\begin{theorem}\label{thm:J-arbitrary-power}
For every set $I$, the dcpo $\J^I$ is $C_{\sigma}$-unique.
\end{theorem}

\begin{proof}
Apply \cref{thm:arbitrary-product} to the constant family $D_i=\J$.
\end{proof}

Taking $I=\N$ answers \cite[Problem~9]{Zhao2022Survey}, repeated as \cite[Problem~1]{Zhao2025Johnstone}, affirmatively. For finite $I$, the conclusion already follows from \cite[Proposition~3.5, Corollary~3.3 and Remark~3.11(i)]{XuZhao2019}; \cref{thm:J-arbitrary-power} extends it to every index set. For every nonempty index set \(I\), the dcpo \(\mathbb J^I\) is non-sober. The proof of \(C_\sigma\)-uniqueness uses local domination and finite-coordinate generators, without asserting local quasicontinuity or bounded sobriety.
For infinite $I$, the proof does not show that $\J^I$ is locally
quasicontinuous, bounded sober or sober. Instead, it proves that $\J^I$ is
locally dominated and reconstructs the product from $C_{\sigma}$-unique
finite-coordinate generators. Thus the countable-power problem is resolved
without a product-sobriety theorem.

\section{\texorpdfstring{$C_\sigma$}{C-sigma}-uniqueness is not preserved under products}\label{sec:product-counterexample}

In this section, we construct a countable sober \(C_{\sigma}\)-unique dcpo \(K\) such that \(K \times \Ichain\), where \(\Ichain=(0,1]\), remains sober but is not \(C_{\sigma}\)-unique. This provides a negative answer to \cite[Problem~8]{Zhao2022Survey}, which asks whether \(C_{\sigma}\)-uniqueness is preserved under products.

Put $X=K\times\Ichain$ and write $p=(p_K,p_{\Ichain})$ for $p\in X$. We construct a proper countable sub-dcpo $\Pprod\subseteq X$ satisfying \eqref{eq:SD} and \eqref{eq:SC}. Marker points recover principal ideals, while private rows provide directed lifting. The main verification is that the union of all construction stages remains closed under directed suprema in $X$.

\subsection{A countable \texorpdfstring{$C_{\sigma}$}{C-sigma}-unique factor}\label{subsec:product-factor}

We first give out the construction of $K$. Let
\(
  \overline{\Npos}=\Npos\cup\{\infty\}
\)
with the usual chain order, where every positive integer is below $\infty$.
For $f:\Npos\to\overline{\Npos}$ put
\[
  \supp(f)=\{i\in\Npos:f(i)<\infty\},
\]
and define
\begin{equation}\label{eq:product-factor}
  K=\{f:\Npos\to\overline{\Npos}:\supp(f)\text{ is finite}\},
\end{equation}
ordered coordinatewise. Denote the constant $\infty$ profile by $\top_K$.

\begin{proposition}\label{prop:product-factor-order}
The dcpo $K$ is countable and has greatest element $\top_K$. Every nonempty
subset $A\subseteq K$ has the coordinatewise supremum
\begin{equation}\label{eq:K-sup}
  \left(\bigvee_K A\right)(i)=\bigvee\{f(i):f\in A\}.
\end{equation}
Finite nonempty infima are also coordinatewise. Moreover, for every
$f\in K$, the principal filter $\up_K f$ contains no infinite antichain.
\end{proposition}

\begin{proof}
There are only countably many finite supports and, on each fixed finite
support, only countably many assignments of positive integers. Thus $K$ is
countable. Choose $f_0\in A$ and let $a$ be the coordinatewise supremum in
\eqref{eq:K-sup}. Since $a\geq f_0$, one has
$\supp(a)\subseteq\supp(f_0)$, so $a\in K$; it is plainly the least upper
bound. The assertion about finite infima is immediate from finite support.

For fixed $f\in K$, restriction to $\supp(f)$ gives
\[
  \up_K f\cong
  \prod_{i\in\supp(f)}[f(i),\infty]_{\overline{\Npos}}.
\]
Every infinite sequence in a well-ordered chain has a nondecreasing
subsequence. Successively passing to subsequences in the finitely many
coordinates gives a coordinatewise nondecreasing subsequence. Hence the
finite product, and therefore $\up_K f$, has no infinite antichain.
\end{proof}

\begin{lemma}\label{lem:down-directed-closed-K}
If $S\subseteq K$ is closed under nonempty directed suprema computed in
$K$, then
\begin{equation}\label{eq:down-S-closed}
  \down_K S\in\Gamma(K).
\end{equation}
The assertion also holds for $S=\varnothing$.
\end{lemma}

\begin{proof}
The empty case is trivial. Let $B\subseteq\down_KS$ be directed and choose
$b_0\in B$. The set $B_0=B\cap\up_Kb_0$ is directed and cofinal in $B$. For every $b\in B_0$, choose $s\in S$ with $b\leq s$; then $s\in T=S\cap\up_Kb_0$. Thus $T\neq\varnothing$ and $B_0\subseteq\down_KT$. Every element of $T$ lies below a maximal element
of $T$: apply Zorn's lemma to $T\cap\up_Ks$, using closure of $S$ under
nonempty directed suprema. By \cref{prop:product-factor-order}, the set of
maximal elements of $T$ is a finite antichain, say
$m_1,\ldots,m_k$. Hence
\(B_0\subseteq\bigcup_{j=1}^k\down_Km_j\).
The right-hand side is a finite union of Scott-closed sets. Therefore
$\bigvee B=\bigvee B_0$ belongs to it and hence to $\down_KS$.
\end{proof}

\begin{proposition}\label{prop:K-sober}
The Scott space $\Sigma K$ is sober.
\end{proposition}

\begin{proof}
Let $C\in\Gamma(K)$ be nonempty and irreducible. Since $C$ is Scott
closed, it is a sub-dcpo of $K$. For each $c\in C$, every chain in
$C\cap\up_Kc$ is directed and has its supremum again in $C\cap\up_Kc$.
Zorn's lemma therefore gives a maximal element of $C$ above $c$, and hence
$C=\down_K\operatorname{Max}(C)$. Every subset of
$\operatorname{Max}(C)$ has only singleton nonempty directed subsets and is
therefore closed under ambient directed suprema. By
\cref{lem:down-directed-closed-K}, the lower closure of every such subset is
Scott closed in $K$.

If $\operatorname{Max}(C)$ contained two distinct elements, choose
$m\in\operatorname{Max}(C)$. Then
\[
  C=\down_Km\cup
    \down_K\bigl(\operatorname{Max}(C)\setminus\{m\}\bigr)
\]
would express $C$ as the union of two proper Scott-closed subsets,
contradicting irreducibility. Thus $C=\down_Km$ for a unique $m\in K$.
\end{proof}

\begin{theorem}\label{thm:K-unique}
The dcpo $K$ is $C_{\sigma}$-unique.
\end{theorem}

\begin{proof}
Since $K$ is sober, it belongs to $\Rtwo$.
By \cref{thm:intrinsic-rigidity}, it suffices to show that
$K$ has no proper $2$-equable sub-dcpo.
Let $S\subseteq K$ be a sub-dcpo for which
restriction $\Gamma(K)\to\Gamma(S)$ is an order isomorphism. Fix $x\in K$.
The set $S_x=S\cap\down_Kx$ is closed under ambient directed suprema, so
\cref{lem:down-directed-closed-K} gives
$F_x=\down_KS_x\in\Gamma(K)$. Moreover,
\(F_x\cap S=S\cap\down_Kx=(\down_Kx)\cap S\).
Injectivity of restriction yields $F_x=\down_Kx$. Since $x\in F_x$, there is
$s\in S_x$ with $x\leq s\leq x$, hence $x=s\in S$. Thus $S=K$.
\end{proof}

\subsection{The retained sub-dcpo}\label{subsec:product-rows}

The chain $\Ichain$ is a continuous dcpo and hence
$C_{\sigma}$-unique \cite[Lemma~3.8]{ZhaoXu2018}.

Choose a sequence $(\theta_i)_{i\in\Npos}$ in
$\mathbb Q\cap(0,1)$ in which every rational in $(0,1)$ occurs infinitely
many times. For $i,n\in\Npos$ define
\begin{equation}\label{eq:beta}
  \beta_i(n)=\theta_i+\frac{1-\theta_i}{n+1}.
\end{equation}
For fixed $i$, the sequence $(\beta_i(n))_n$ is strictly decreasing to
$\theta_i$. Define the initial marker set
\begin{equation}\label{eq:P0-product}
  \Pstageprod{0}=\{(f,\beta_i(n)):f\in K,\ i,n\in\Npos,\ f(i)=n\}.
\end{equation}
Let $X_{\mathbb Q}=K\times(\mathbb Q\cap(0,1))$ and let $\Tprod$ be the
countable set of triples
\begin{equation}\label{eq:T-product}
\begin{aligned}
  \Tprod=\{(a,H,q):\;&a\in K,\quad H\subseteq\Npos\text{ finite},
  \quad \supp(a)\subseteq H,\\
  &q=(b,r)\in X_{\mathbb Q},\quad b\leq a\}.
\end{aligned}
\end{equation}
Choose once and for all an injection
\[
  \hprod:\Tprod\longrightarrow\Npos
\]
such that, for $\tau=(a,H,(b,r))$,
\begin{equation}\label{eq:h-product}
  \hprod(\tau)\notin H\cup\supp(b).
\end{equation}
This is possible by enumerating $\Tprod$ and avoiding the finite
forbidden set together with the previously used private coordinates.

For $a\in K$, finite $H\supseteq\supp(a)$, and $n\in\Npos$, define
\begin{equation}\label{eq:u-product}
  u^{\times}_{a,H,n}(i)=
  \begin{cases}
    \min\{a(i),n\},&i\in H,\\
    \infty,&i\notin H,
  \end{cases}
\end{equation}
and, for $k,n\in\Npos$,
\begin{equation}\label{eq:z-product}
  z^{\times}_{k,n}(i)=
  \begin{cases}
    n,&i=k,\\
    \infty,&i\neq k.
  \end{cases}
\end{equation}
If $\tau=(a,H,q)\in\Tprod$ and $q=(b,r)$, define
\begin{equation}\label{eq:row-product}
  \rowprod(\tau,n)=
  \left(
    b\wedge u^{\times}_{a,H,n}\wedge z^{\times}_{\hprod(\tau),n},
    \frac{n}{n+1}r
  \right).
\end{equation}
The meet is coordinatewise minimum: $b$ fixes the endpoint, $u^{\times}_{a,H,n}$ controls ambient approximation, and $z^{\times}_{\hprod(\tau),n}$ records the row at a private coordinate.

\begin{lemma}\label{lem:product-rows}
For every $\tau=(a,H,q)\in\Tprod$, the sequence
$(\rowprod(\tau,n))_{n\in\Npos}$ is strictly increasing and
\begin{equation}\label{eq:row-product-sup}
  \bigvee_{n\in\Npos}\rowprod(\tau,n)=q.
\end{equation}
At the private coordinate $k=\hprod(\tau)$ one has
\begin{equation}\label{eq:row-private-coordinate}
  \rowprod(\tau,n)_K(k)=n.
\end{equation}
\end{lemma}

\begin{proof}
Write $q=(b,r)$ and $k=\hprod(\tau)$. By \eqref{eq:h-product}, both $b(k)$ and
$u^{\times}_{a,H,n}(k)$ equal $\infty$, so \eqref{eq:row-private-coordinate} holds.
The first coordinate of every row has finite support contained in
$\supp(b)\cup H\cup\{k\}$, and the second coordinate belongs to
$\mathbb Q\cap(0,r)$. All coordinates are nondecreasing in $n$, while the
second coordinate is strictly increasing.

On $H$, the first coordinate is $\min\{b(i),n\}$ because $b\leq a$; outside
$H\cup\{k\}$ it is $b(i)$; and at $k$ it increases to
$\infty=b(k)$. The second coordinate increases to $r$. Thus the supremum is
$q$.
\end{proof}

Define recursively
\begin{equation}\label{eq:P-product}
\begin{aligned}
  \Pstageprod{m+1}
  &=\Pstageprod{m}\cup
    \{\rowprod(\tau,n):\tau=(a,H,q)\in\Tprod,\ q\in \Pstageprod{m},\ n\in\Npos\},\\
  \Pprod
  &=\bigcup_{m\in\N}\Pstageprod{m}.
\end{aligned}
\end{equation}

\begin{proposition}\label{prop:product-P-basic}
The set $\Pprod$ is countable. Every $p\in \Pprod$ satisfies
\begin{equation}\label{eq:P-product-basic}
  0<p_{\Ichain}<1,\qquad p_{\Ichain}\in\mathbb Q,
  \qquad \supp(p_K)\neq\varnothing.
\end{equation}
Consequently
\begin{equation}\label{eq:P-product-no-top-profile}
  \Pprod\cap(\{\top_K\}\times\Ichain)=\varnothing.
\end{equation}
Moreover, $\Pprod$ has no greatest element.
\end{proposition}

\begin{proof}
Countability follows from countability of $\Pstageprod{0}$, $\Tprod$ and $\Npos$.
The marker points satisfy \eqref{eq:P-product-basic}; every new row has
rational height in $(0,1)$ and nonempty support by
\eqref{eq:row-private-coordinate}. Induction gives
\eqref{eq:P-product-basic} and \eqref{eq:P-product-no-top-profile}.

It remains to show that $\Pprod$ has no greatest element. Fix
$p\in \Pprod$. By \eqref{eq:P-product-basic}, the finite set
$\supp(p_K)$ is nonempty. Choose
$r\in\Npos\setminus\supp(p_K)$ and put
\(q=(z^{\times}_{r,1},\beta_r(1))\).
Then $q\in \Pstageprod{0}\subseteq \Pprod$. Choose any
$s\in\supp(p_K)$. Since $r\notin\supp(p_K)$, we have $s\neq r$, and hence
\(q_K(s)=\infty>p_K(s)\).
Thus $q\nleq p$. Since every $p\in \Pprod$ fails to dominate some
$q\in \Pprod$, the poset $\Pprod$ has no greatest element.
\end{proof}

\begin{theorem}\label{thm:product-P-dcpo}
 $\Pprod$ is a sub-dcpo of $X$.
\end{theorem}

\begin{proof}
Give each $p\in\Pprod$ a fixed certificate, using its least construction stage
\[
  \operatorname{rk}_{\times}(p)=\min\{m:p\in\Pstageprod{m}\}.
\]
For rank zero, choose a marker representation $p=(f,\beta_i(n))$ with $f(i)=n$ and use $(0,i)$. Otherwise choose $p=\rowprod(\tau,n)$ from its first stage, write $\tau=(a,H,q_\tau)$, and use $(1,\hprod(\tau))$. The endpoint $q_\tau$ belongs to the preceding stage. In both cases the certificate coordinate belongs to $\supp(p_K)$. For a row certificate, injectivity of $\hprod$ determines $\tau$, and the value at that coordinate determines $n$. A private coordinate may occur in other supports; only its use as a row certificate must identify a unique row family.

Let $p_0\leq p_1\leq\cdots$ be increasing in $\Pprod$ and let $p$ be
its supremum in $X$. Since supports decrease along increasing sequences,
\[
  \supp((p_j)_K)\subseteq\supp((p_0)_K)
  \qquad(j\in\N).
\]
Only finitely many certificate types and coordinates can therefore occur.
Choose an infinite subsequence with one fixed certificate. Its index set is
unbounded, so the subsequence is cofinal.

In the marker case, the subsequence has the form
$(f_j,\beta_i(n_j))$ with $f_j(i)=n_j$. Monotonicity of the first
coordinates makes $(n_j)$ nondecreasing, while monotonicity of the second
coordinates and strict decrease of $\beta_i$ force $(n_j)$ to be constant,
say $n_j=n$. The cofinal subsequence therefore has constant second
coordinate $\beta_i(n)$ and constant $i$th first coordinate $n$. Its ambient
supremum has the same two coordinates and hence is again a marker in $\Pstageprod{0}$.

In the row case, injectivity of $\hprod$ forces all points in the cofinal
subsequence to come from one triple $\tau$, so they are
$\rowprod(\tau,n_j)$. The private coordinate makes $(n_j)$ nondecreasing. If it
is bounded, it is eventually constant and the supremum is a retained row.
If it is unbounded, \cref{lem:product-rows} gives supremum $q_\tau$, which
belongs to an earlier construction stage. Hence every increasing sequence
in $\Pprod$ has its ambient supremum in $\Pprod$.

Finally, $\Pprod$ is countable, so every directed subset has an
increasing cofinal sequence. Its ambient supremum therefore belongs to
$\Pprod$ and is also its supremum in $\Pprod$.
\end{proof}

\subsection{Preservation of the Scott-closed-set lattice}\label{subsec:product-trace}

It remains to verify the two conditions from \cref{thm:closure-characterization}. The marker set already yields Skula-density; the private rows are designed precisely to give directed lifting and hence closure compatibility.

\begin{lemma}\label{lem:product-Skula-density}
For every $x\in X$,
\begin{equation}\label{eq:product-TG}
  \cl_{\sigma(X)}(\Pstageprod{0}\cap\down_Xx)=\down_Xx.
\end{equation}
In particular, $\Pprod\subseteq X$ satisfies Skula-density condition
\textup{(SD)}.
\end{lemma}

\begin{proof}
Write $x=(a,t)$ and let
$C=\cl_{\sigma(X)}(\Pstageprod{0}\cap\down_Xx)$. Clearly $C\subseteq\down_Xx$. Fix
$c\in\mathbb Q\cap(0,t)$. Since $\supp(a)$ is finite and the value $c$
occurs infinitely often in $(\theta_i)$, choose
$i\notin\supp(a)$ with $\theta_i=c$. Let $a^{i,n}$ agree with $a$ except
that $a^{i,n}(i)=n$. Then $(a^{i,n})_n$ is increasing with supremum $a$.

For all sufficiently large $n$, one has $\beta_i(n)\leq t$, and hence
\((a^{i,n},\beta_i(n))\in \Pstageprod{0}\cap\down_X(a,t)\subseteq C\).
Since $c<\beta_i(n)$ and $C$ is lower, $(a^{i,n},c)\in C$. Taking the
directed supremum over the tail in $n$ gives $(a,c)\in C$. Finally,
\((a,t)=\bigvee\{(a,c):c\in\mathbb Q,\ 0<c<t\}\),
so $x\in C$. Lowerness then gives $\down_Xx\subseteq C$.
\end{proof}

\begin{proposition}\label{prop:product-directed-lifting}
The inclusion $\Pprod\subseteq X$ has directed lifting.
\end{proposition}

\begin{proof}
Let $A\subseteq X$ be directed, write $\bigvee_XA=(a,t)$, and let
$q=(b,r)\in \Pprod$ satisfy $q\leq(a,t)$. By
\cref{prop:product-P-basic}, $r\in\mathbb Q\cap(0,1)$.

Choose $(f_0,s_0)\in A$, put $H=\supp(f_0)$, and let
$A'=A\cap\up_X(f_0,s_0)$. Then $A'$ is directed and cofinal,
$\supp(a)\subseteq H$, and every first coordinate of a point of $A'$ is
$\infty$ outside $H$. Hence $\tau=(a,H,q)\in\Tprod$.

For each $n\in\Npos$ and $i\in H$, the finite element
$\min\{a(i),n\}$ is compact in $\overline{\Npos}$ and lies below the
supremum of the $i$th coordinates of $A'$. Together with
$\frac{n}{n+1}r<t$, these finitely many requirements and directedness yield
$y_n=(f_n,s_n)\in A'$ such that
\[
  u^{\times}_{a,H,n}\leq f_n,
  \qquad
  \frac{n}{n+1}r\leq s_n.
\]
Put $q_n=\rowprod(\tau,n)$. By \eqref{eq:P-product}, each $q_n$ belongs to
$\Pprod$, and \cref{lem:product-rows} shows that $(q_n)_n$ is increasing with supremum $q$. Moreover,
\eqref{eq:row-product} yields $q_n\leq y_n$. Thus
$\{q_n:n\in\Npos\}$ witnesses directed lifting.
\end{proof}

\begin{theorem}\label{thm:product-Gamma-isomorphism}
The restriction map is an order isomorphism
\begin{equation}\label{eq:product-Gamma-iso}
  \rho:\Gamma(X)\longrightarrow\Gamma(\Pprod),
  \qquad H\longmapsto H\cap \Pprod.
\end{equation}
Its inverse is $C\mapsto\cl_{\sigma(X)}(C)$.
\end{theorem}

\begin{proof}
\cref{lem:product-Skula-density} gives \eqref{eq:SD}. By
\cref{prop:product-directed-lifting,prop:directed-lifting-bridge-density},
directed lifting gives \eqref{eq:SC}. Apply \cref{thm:closure-characterization}.
\end{proof}

\subsection{The counterexample}\label{subsec:product-main}

\begin{proposition}\label{prop:product-sober}
The Scott topology of $X=K\times\Ichain$ is the product topology of
$\Sigma K$ and $\Sigma\Ichain$. Consequently, $\Sigma X$ is sober.
\end{proposition}

\begin{proof}
Every basic open rectangle $U\times V$, with $U\in\sigma(K)$ and
$V\in\sigma(\Ichain)$, is Scott open in the product dcpo. Indeed, it is an upper
set. If $E\subseteq K\times\Ichain$ is directed and
$\bigvee E\in U\times V$, Scott openness of $U$ and $V$ gives
$e_1,e_2\in E$ whose first and second coordinates lie in $U$ and $V$,
respectively. Directedness yields $e\in E$ with $e_1,e_2\leq e$, and then
$e\in U\times V$. Hence the product topology is contained in $\sigma(X)$.

For the reverse inclusion, let $W\in\sigma(X)$ and $(a,t)\in W$. Since
\((a,t)=\bigvee\{(a,s):0<s<t\}\),
Scott openness gives $0<s<t$ with $(a,s)\in W$. For this fixed $s$, the map
\(\iota_s:K\longrightarrow K\times\Ichain,\qquad b\longmapsto(b,s)\),
preserves directed suprema and is therefore Scott continuous. Consequently
its inverse image
\(W_s=\iota_s^{-1}(W)=\{b\in K:(b,s)\in W\}\)
is Scott open in $K$. The interval $(s,1]$ is Scott open in the chain
$\Ichain$. Since $W$ is an upper set,
\((a,t)\in W_s\times(s,1]\subseteq W\).
Thus every Scott-open neighbourhood contains a product-open neighbourhood,
and the two topologies agree. Both $\Sigma K$ and $\Sigma\Ichain$ are sober: the former by \cref{prop:K-sober}, the latter because $\Ichain$ is continuous. Topological products of sober spaces are sober \cite{goubault2013}, so $\Sigma X$ is sober.
\end{proof}

\begin{theorem}\label{thm:product-counterexample}
There exists a countable sober $C_{\sigma}$-unique dcpo $K$ such that,
for the continuous chain $\Ichain=(0,1]$, the product $K\times\Ichain$ is sober but
not $C_{\sigma}$-unique. More precisely, $K\times\Ichain$ has a proper
countable $2$-equable sub-dcpo $\Pprod$ with
\begin{equation}\label{eq:product-counterexample}
  \Gamma(K\times\Ichain)\cong\Gamma(\Pprod),
  \qquad
  K\times\Ichain\not\cong \Pprod.
\end{equation}
Thus the class of $C_{\sigma}$-unique dcpos is not closed under binary
products.
\end{theorem}

\begin{proof}
The dcpo $K$ is countable, sober and $C_{\sigma}$-unique by
\cref{prop:product-factor-order,prop:K-sober,thm:K-unique}, while
$\Ichain$ is continuous and $C_{\sigma}$-unique. The product is sober by \cref{prop:product-sober}. By
\cref{prop:product-P-basic,thm:product-P-dcpo,thm:product-Gamma-isomorphism},
$\Pprod$ is a proper countable $2$-equable sub-dcpo of
$X=K\times\Ichain$ and has no greatest element. Since $X$ has greatest
element $(\top_K,1)$, one has $\Pprod\not\cong X$, so
\eqref{eq:product-counterexample} shows that $X$ is not
$C_{\sigma}$-unique. 
\end{proof}

By \cref{cor:sober-equable-completions}, the inclusion $\Pprod\hookrightarrow X$ is the two-element reflection of $\Pprod$ and the sobrification map of $\Sigma\Pprod$. Hence
\[
  K\times\Ichain\cong L_2\Pprod\cong\widehat{\Pprod},
  \qquad \Pprod\notin\Rtwo,
\]
and $\Pprod$ is not sober. Only $K$ and $\Pprod$ are countable; $\Ichain$ and $X$ are uncountable.

\section{A sober countable frame that is not \texorpdfstring{$C_{\sigma}$}{C-sigma}-unique}\label{sec:counterexample}

Problems~5 and~6 of \cite{Zhao2022Survey} ask whether every complete lattice, or at least every sober countable complete lattice, is $C_{\sigma}$-unique. We answer both negatively by constructing a countable frame $F$ with a proper countable $2$-equable sub-dcpo $\Pframe$. The finite-support and private-row mechanism is the same as in Section~5. Here an antichain in the coordinate frame replaces the opposing monotonicities of marker coordinates and heights; \cref{tab:construction-comparison} summarizes the two implementations.

\subsection{The coordinate frame and marker system}\label{subsec:countable-frames}

A \emph{frame} is a complete lattice $F$ satisfying
\[
  x\wedge\bigvee_{i\in I}y_i
  =\bigvee_{i\in I}(x\wedge y_i)
\]
for all $x$ and all families $(y_i)_{i\in I}$. A quasi-order is a
\emph{well-quasi-order}, or wqo, if every infinite sequence
$x_0,x_1,\ldots$ has $i<j$ with $x_i\leq x_j$.

Every countable frame is sober in its Scott topology, since frames are meet-continuous complete lattices and \cite[Theorem~1.1]{XuJi2026} applies. Finitary distributivity alone does not suffice: a countable complete distributive lattice can have a non-sober Scott space \cite{MiaoXiLiZhao2023}.

Let
\[
  \RR=\{(m,n)\in\N^2:m<n\}
\]
and define
\begin{equation}\label{eq:rado}
  (m,n)\preceq(m',n')
  \quad\Longleftrightarrow\quad
  \bigl(m=m'\text{ and }n\leq n'\bigr)\text{ or }n<m'.
\end{equation}
This is the standard Rado order; it is a partial order and a wqo
\cite{Rado1954,Pequignot2017}.

Let
\[
  \Acoord=\Down(\RR)
\]
be the set of lower subsets of $\RR$, ordered by inclusion, with
$\Abot=\varnothing$ and $\Atop=\RR$.

\begin{proposition}\label{prop:coordinate-frame}
The poset $\Acoord$ is a countable algebraic frame. Arbitrary joins and meets are
set-theoretic unions and intersections.
\end{proposition}

\begin{proof}
Lower sets are closed under arbitrary unions and intersections, and
intersection distributes over arbitrary union. Hence $\Acoord$ is a frame.

Let $U$ be an upper set of a wqo. Its minimal elements form an antichain and
are therefore finite. Every $u\in U$ lies above a minimal member of $U$:
otherwise one can recursively choose
$u=u_0>u_1>u_2>\cdots$ inside $U$, contradicting well-foundedness. Hence
$U$ is generated by its finite set of minimal elements. It follows that every
lower set has the form $\RR\setminus\up F$ for some finite
$F\subseteq\RR$. Since $\RR$ is countable, only countably many such $F$
occur, and $\Acoord$ is countable.

Every lower set $U$ is the directed union of $\down_{\RR}F$, where $F$
ranges over the finite subsets of $U$. Such a finitely generated lower set
is compact: if $\down_{\RR}F$ is contained in a directed union of lower
sets, finitely many members cover $F$, and one common upper member of the
directed family contains all of $\down_{\RR}F$. Hence $\Acoord$ is algebraic.
\end{proof}

For $n,m\in\N$, define
\begin{equation}\label{eq:tn-Dm}
  t_n=\{(i,j)\in\RR:j\leq n+1\},\qquad
  D_m=\down_{\RR}\{(m,\ell):m<\ell\},\qquad
  e_n=D_{n+2}.
\end{equation}
\begin{lemma}\label{lem:markers}
The following hold.
\begin{enumerate}[label=\textup{(\roman*)},leftmargin=2.2em]
\item Each $t_n$ is finite and lower,
      $t_0\subsetneq t_1\subsetneq\cdots$, and
      $\bigcup_nt_n=\Atop$.
\item
      \(
        D_m=\{(i,j)\in\RR:j<m\}\cup\{(m,j):m<j\}.
      \)
\item $t_n\subseteq e_n\subsetneq\Atop$ for every $n$.
\item $(e_n)_{n<\omega}$ is an antichain in $\Acoord$.
\end{enumerate}
\end{lemma}

\begin{proof}
The assertions about $t_n$ are immediate from \eqref{eq:rado}. An element
$(i,j)$ lies below some $(m,\ell)$ with $m<\ell$ exactly when $j<m$ or
$i=m$, proving the formula for $D_m$. If $(i,j)\in t_n$, then
$j\leq n+1<n+2$, so $(i,j)\in D_{n+2}=e_n$. Moreover, $e_n\subsetneq\Atop$, since $(n+3,n+4)\notin e_n$.

For $a<b$,
\[
  (a,b)\in D_a\setminus D_b,
  \qquad
  (b,b+1)\in D_b\setminus D_a.
\]
Thus the $D_m$, and hence the shifted family $(e_n)$, are pairwise
incomparable.
\end{proof}

The antichain $(e_n)$ lies in $\Acoord=\Down(\RR)$, not in $\RR$ itself, so its existence does not contradict the wqo property of the Rado order.

\subsection{The countable profile frame}\label{subsec:profile-frame}

Fix a countably infinite coordinate set $S$. Put
\[
  F_0=\left\{f:S\to\Acoord:
  \supp(f)=\{s:f(s)\neq\Atop\}\text{ is finite}\right\}.
\]
Order $F_0$ coordinatewise by inclusion. It has no least element: for any profile $f$, choose $s\notin\supp(f)$ and replace its value at $s$ by $\Abot$, obtaining a strictly smaller profile. Adjoin a new element $\Fbot$ below all profiles and put
\[
  F=\{\Fbot\}\cup F_0.
\]
The constant profile with value $\Atop$ is denoted by $\Ftop$. The new bottom $\Fbot$ is not a profile; in particular, it is distinct from every profile having one or more coordinates equal to $\Abot$.

\begin{theorem}\label{thm:F-frame}
The poset $F$ is a countable frame. For every nonempty family
$\mathcal B\subseteq F_0$,
\begin{equation}\label{eq:profile-join}
  \left(\bigvee\mathcal B\right)(s)
  =\bigcup_{f\in\mathcal B}f(s).
\end{equation}
If $\mathcal B\subseteq F_0$ and
$U=\bigcup_{f\in\mathcal B}\supp(f)$ is finite, then
\begin{equation}\label{eq:profile-meet}
  \left(\bigwedge\mathcal B\right)(s)
  =\bigcap_{f\in\mathcal B}f(s),
\end{equation}
where the empty intersection is $\Atop$; if $U$ is infinite, then
$\bigwedge\mathcal B=\Fbot$. The join of the empty family is $\Fbot$,
bottom terms may be discarded from a nonempty join, and any meet containing
$\Fbot$ is $\Fbot$.
\end{theorem}

\begin{proof}
Countability follows because $S$ and $\Acoord$ are countable and every profile
has finite support. For \eqref{eq:profile-join}, fix
$f_0\in\mathcal B$. Outside $\supp(f_0)$ the coordinate union is $\Atop$, so
it again has finite support and is plainly the least upper bound. If $U$ is
finite, the coordinate intersection in \eqref{eq:profile-meet} is the greatest
lower bound in $F_0$. If $U$ is infinite, a nonbottom common lower bound would
have non-top value at every coordinate of $U$, contradicting finite support;
thus the meet is $\Fbot$.

It remains to verify the frame law. The empty family and the case in which
all $g_i$ are bottom are immediate. Otherwise discard the bottom terms.
Finite meets are then coordinatewise intersections, and for $f\in F_0$,
\[
  \left(f\wedge\bigvee_i g_i\right)(s)
  =f(s)\cap\bigcup_i g_i(s)
  =\bigcup_i(f(s)\cap g_i(s))
  =\left(\bigvee_i(f\wedge g_i)\right)(s).
\]
The case $f=\Fbot$ and the empty family are immediate.
\end{proof}

\begin{remark}\label{rem:profile-not-algebraic}
Although $\Acoord$ is algebraic, $F$ is not. For a nonbottom profile $f$, choose $r\notin\supp(f)$ and let $f_n$ agree with $f$ except that $f_n(r)=t_n$. Then $(f_n)_n$ is increasing, every $f_n<f$, and $\bigvee_n f_n=f$. Hence no nonbottom element of $F$ is compact.
\end{remark}

\subsection{The retained sub-dcpo}\label{subsec:private-rows}

We construct $\Pframe\subseteq F$ using finite-coordinate approximants and private rows, then prove that increasing sequences have their ambient suprema in $\Pframe$.

Let $a\in F_0$ and let $H\subseteq S$ be finite with
$\supp(a)\subseteq H$. Write $\tau=(a,H)$ and define
\begin{equation}\label{eq:canonical-u}
  u^{F}_{\tau,n}(s)=
  \begin{cases}
    a(s)\cap t_n,&s\in H,\\
    \Atop,&s\notin H.
  \end{cases}
\end{equation}
\begin{lemma}\label{lem:canonical-approximants}
The sequence $(u^{F}_{\tau,n})$ is increasing and has supremum $a$. Moreover,
if $(y_k)$ is an increasing sequence of nonbottom profiles with supremum
$a$, then its supports are eventually a fixed finite set $H$, and for every
$n$ some $k$ satisfies
\begin{equation}\label{eq:u-below}
  u^{F}_{(a,H),n}\leq y_k.
\end{equation}
\end{lemma}

\begin{proof}
The first assertion follows coordinatewise from $\bigcup_nt_n=\Atop$. Supports decrease along an increasing sequence and hence eventually stabilize at a finite set $H$. One has $\supp(a)\subseteq H$, but equality is not required: a coordinate may reach $\Atop$ only at the supremum. At each coordinate,
$a(s)=\bigcup_ky_k(s)$. For fixed $n$ and $s\in H$, the set
$u^{F}_{(a,H),n}(s)$ is finite and contained in $a(s)$, so it is contained in
some sufficiently late $y_k(s)$. One common $k$ works for all coordinates
in finite $H$; outside $H$ both profiles have value $\Atop$.
\end{proof}

Let $\Tframe$ be the countable set of all pairs $\tau=(a,H)$ above. For
every eligible pair $(\tau,q)$, where $\tau=(a,H)$, $q\in F_0$ and
$q\leq a$, choose injectively a private coordinate
\begin{equation}\label{eq:private-coordinate}
  \hframe(\tau,q)\in S\setminus(H\cup\supp(q)).
\end{equation}
Enumerate the eligible pairs. At each stage choose the new coordinate
outside $H\cup\supp(q)$ and outside the finite set of previously chosen
private coordinates. This makes $\hframe$ injective.

For $r\in S$, let $z^{F}_{r,n}$ be the profile with value $t_n$ at $r$ and
$\Atop$ elsewhere. Define
\begin{equation}\label{eq:row}
  \rowframe(\tau,q,n)=q\wedge u^{F}_{\tau,n}\wedge z^{F}_{\hframe(\tau,q),n}.
\end{equation}
\begin{lemma}\label{lem:rows}
For every eligible pair $(\tau,q)$, the rows increase to $q$:
\[
  \bigvee_n\rowframe(\tau,q,n)=q.
\]
At the private coordinate $k=\hframe(\tau,q)$,
$\rowframe(\tau,q,n)(k)=t_n$.
\end{lemma}

\begin{proof}
Put $k=\hframe(\tau,q)$. The meet in \eqref{eq:row} is coordinatewise intersection. At $k$, both $q$ and $u^{F}_{\tau,n}$ have value $\Atop$, so the row has value $t_n$. At every other coordinate,
\[
  \bigcup_n\rowframe(\tau,q,n)(s)
  =q(s)\cap\bigcup_nu^{F}_{\tau,n}(s)
  =q(s)\cap a(s)=q(s),
\]
using $q\leq a$. At the private coordinate the union of the $t_n$ is
$\Atop=q(k)$.
\end{proof}

Define
\begin{equation}\label{eq:P0}
  \Pstageframe{0}=\{\Fbot\}\cup
  \{p\in F_0:\text{for some }s\in S,i\in\N,\ p(s)=e_i\}.
\end{equation}
Inductively put
\begin{equation}\label{eq:Pm}
\begin{split}
  \Pstageframe{m+1}=\Pstageframe{m}\cup\{\rowframe(\tau,q,n):
  &\ \tau=(a_\tau,H_\tau)\in\Tframe,\\
  &\ q\in \Pstageframe{m}\cap F_0,\ q\leq a_\tau,\ n\in\N\},
\end{split}
\end{equation}
and let $\Pframe=\bigcup_m\Pstageframe{m}$. The set $\Pframe$ is countable.

As in Section~5, each nonbottom point will receive a marker or row certificate. Finite support reduces any increasing sequence to a cofinal subsequence with one fixed certificate type and coordinate.

\begin{theorem}\label{thm:P-dcpo}
The inherited order makes $\Pframe$ a dcpo, and all directed suprema in $\Pframe$ agree
with those in $F$. Thus $\Pframe\subseteq F$ is a sub-dcpo.
\end{theorem}

\begin{proof}
Give every nonbottom retained profile a certificate. If it has a coordinate
with value $e_i$, choose one such coordinate and call it a marker
certificate. Otherwise choose a row representation
$p=\rowframe(\tau,q,n)$ from the least stage at which $p$ appears and use its
private coordinate as a row certificate. At that coordinate the value is
$t_n$. Injectivity of $\hframe$ means that a row-certificate coordinate determines
$(\tau,q)$, while the value $t_n$ determines the row level $n$ because the
sequence $(t_n)$ is strictly increasing.  Thus, after a certificate type and
coordinate have been fixed, the corresponding subsequence is controlled by a
single marker family or a single private row family.

Let $p_0\leq p_1\leq\cdots$ be increasing in $\Pframe$. If every term equals $\Fbot$, the conclusion is immediate. Otherwise discard the finitely many terms before the first nonbottom term and reindex. Let $p=\bigvee_F p_k$. Supports decrease, so
$\supp(p_k)\subseteq\supp(p_0)$ for all $k$, and therefore
$\supp(p)\subseteq\supp(p_0)$; in particular $p\in F_0$. Every certificate
coordinate also belongs to the finite set $\supp(p_0)$. There are only two
certificate types and only finitely many possible certificate coordinates.
Hence one certificate type and one coordinate occur on an infinite
subsequence. Its set of indices is unbounded, so the subsequence is cofinal
in the original sequence and has the same ambient supremum $p$.

In the marker case, at a fixed coordinate $s$ we have
$p_{k_j}(s)=e_{i_j}$. Comparability gives
$e_{i_j}\subseteq e_{i_{j+1}}$, and the antichain property forces all
$i_j$ to equal one $i$. For $k_j\leq k\leq k_{j+1}$,
\(e_i=p_{k_j}(s)\subseteq p_k(s)\subseteq p_{k_{j+1}}(s)=e_i\).
Thus the original sequence is eventually $e_i$ at $s$. Consequently the
ambient supremum $p$ satisfies $p(s)=e_i$; since $p\in F_0$, it follows from
\eqref{eq:P0} that $p\in \Pstageframe{0}$.

In the row case, the fixed private coordinate determines one pair
$(\tau,q)$, so the cofinal subsequence is $\rowframe(\tau,q,n_j)$. At the private
coordinate its values are $t_{n_j}$; comparability makes $(n_j)$
nondecreasing. If it is bounded, it is eventually constant, so the cofinal
subsequence, and hence the original sequence, has a retained row as its
supremum. If it is unbounded, it is cofinal in the full row sequence and has
supremum $q$ by \cref{lem:rows}; this $q$ lies in an earlier construction
stage. Hence every increasing sequence has its $F$-supremum in $\Pframe$.

Since $\Pframe$ is countable, every directed subset has an increasing cofinal
sequence. Its supremum therefore belongs to $\Pframe$ and is the same in $\Pframe$ and
$F$.
\end{proof}

\subsection{Preservation of the Scott-closed-set lattice}\label{subsec:trace-isomorphism}

The private rows and marker system now verify the two sufficient conditions
introduced in \cref{def:directed-lifting-bridge-density}.

\begin{proposition}\label{prop:frame-approximation-properties}
The inclusion $\Pframe\subseteq F$ has directed lifting and bridge density.
\end{proposition}

\begin{proof}
For directed lifting, let $E\subseteq F$ be directed, put $a=\bigvee E$, and
let $q\in \Pframe$ satisfy $q\leq a$. If $q=\Fbot$, use $\{\Fbot\}$. Otherwise,
$F$ is countable by \cref{thm:F-frame}, so choose an increasing cofinal
sequence $(y_k)$ in $E$ and delete its initial bottom terms. The supports
stabilize at a finite set $H$, necessarily containing $\supp(a)$. By
\cref{lem:canonical-approximants}, for each $n$ there is $k(n)$ with
$u^{F}_{(a,H),n}\leq y_{k(n)}$. Since $q\leq a$, the pair $((a,H),q)$ is eligible.
Let $b_n=\rowframe((a,H),q,n)$. If $q\in \Pstageframe{m}$, then $b_n\in \Pstageframe{m+1}$ by
\eqref{eq:Pm}; moreover, $b_n\leq y_{k(n)}$ and $\bigvee_n b_n=q$ by
\cref{lem:rows}. Thus $(b_n)$ witnesses directed lifting.

For bridge density, the bottom element is immediate. Given $a\in F_0$, choose
$r\notin\supp(a)$ and define
\[
  v_n(s)=
  \begin{cases}
    t_n,&s=r,\\
    a(s),&s\neq r,
  \end{cases}
  \qquad
  c_n(s)=
  \begin{cases}
    e_n,&s=r,\\
    a(s),&s\neq r.
  \end{cases}
\]
Then $v_n\leq c_n\leq a$, the point $c_n$ belongs to $\Pstageframe{0}$, and
$\bigvee_n v_n=a$. Hence $(v_n,c_n)$ witnesses bridge density. This is the
only place where the inclusions $t_n\subseteq e_n$ are used.
\end{proof}

\begin{theorem}\label{thm:Gamma-isomorphism}
The restriction map is an order isomorphism
\[
  \rho:\Gamma(F)\longrightarrow\Gamma(\Pframe),
  \qquad \rho(C)=C\cap \Pframe.
\]
Its inverse is $C\mapsto\cl_{\sigma(F)}(C)$.
\end{theorem}

\begin{proof}
Apply \cref{prop:directed-lifting-bridge-density} to
\cref{thm:P-dcpo,prop:frame-approximation-properties}.
\end{proof}

\begin{proposition}\label{prop:P-not-isomorphic-F}
The dcpo $\Pframe$ has no greatest element. Therefore $\Pframe\not\cong F$.
\end{proposition}

\begin{proof}
First, $\Ftop\notin \Pframe$: it has no marker coordinate, so it is not in $\Pstageframe{0}$,
and every nonbottom row has value $t_n\neq\Atop$ at its private coordinate.
This persists through all stages $\Pstageframe{m}$.

If $p=\Fbot$, any nonbottom marker profile is not below $p$. Let
$p\neq\Fbot$. Since $p\neq\Ftop$, its support is nonempty and finite. Choose
$r\notin\supp(p)$ and let $m_r$ have value $e_0$ at $r$ and $\Atop$
elsewhere. Then $m_r\in \Pstageframe{0}$, but for any $s\in\supp(p)$ one has
$m_r(s)=\Atop\not\subseteq p(s)$, so $m_r\not\leq p$. Thus $\Pframe$ has no
greatest element, whereas $F$ has greatest element $\Ftop$.
\end{proof}

\begin{table}[!htbp]
\centering
\caption{The common finite-support mechanism in the two counterexamples.}
\label{tab:construction-comparison}
\small
\renewcommand{\arraystretch}{1.15}
\begin{tabularx}{\textwidth}{@{}>{\raggedright\arraybackslash}p{0.21\textwidth}>{\raggedright\arraybackslash}X>{\raggedright\arraybackslash}X@{}}
\toprule
Role & Product construction (Section~5) & Frame construction (Section~6) \\
\midrule
Ambient dcpo & $K\times\Ichain$ & $F=\{\Fbot\}\cup F_0$ \\
Retained stages & $\Pstageprod{m}$, with union $\Pprod$ & $\Pstageframe{m}$, with union $\Pframe$ \\
Marker stability & Opposite monotonicities of $n$ and $\beta_i(n)$ & The antichain $(e_n)$ \\
Private row level & Value $n$ at $\hprod(\tau)$ & Value $t_n$ at $\hframe(\tau,q)$ \\
Unbounded row levels & Supremum is a previously retained endpoint & Supremum is a previously retained endpoint \\
Closure identities & Marker density and directed lifting & Bridge density and directed lifting \\
\bottomrule
\end{tabularx}
\end{table}
\FloatBarrier

\begin{theorem}\label{thm:countable-frame-counterexample}
There exist a countable frame $F$ and a countable dcpo $\Pframe$ such that
\[
  \Gamma(F)\cong\Gamma(\Pframe),
  \qquad F\not\cong \Pframe.
\]
Moreover, $\Sigma F$ is sober. Hence a sober countable frame need not be
$C_{\sigma}$-unique.
\end{theorem}

\begin{proof}
The poset $F$ is a countable frame by \cref{thm:F-frame}, and $\Pframe$ is a
countable dcpo by \cref{thm:P-dcpo}. The isomorphism of Scott-closed-set
lattices is \cref{thm:Gamma-isomorphism}, while
\cref{prop:P-not-isomorphic-F} gives $F\not\cong \Pframe$. Finally, every
countable frame is Scott sober \cite[Theorem~1.1]{XuJi2026}.
\end{proof}

This answers \cite[Problems~5 and~6]{Zhao2022Survey} negatively, even for countable complete Heyting algebras.

\section{Conclusion and further questions}\label{sec:conclusion}

The positive and negative results use complementary reconstruction arguments. Local domination and directed generation yield the arbitrary-product theorem and the $C_{\sigma}$-uniqueness of $\J^I$. Independently, the two-element reflection and the closure criterion identify the proper sub-dcpos used in the counterexamples. \Cref{tab:questions} records the answers to the cited questions.

\begin{table}[!htbp]
\centering
\caption{Questions answered by the main results.}
\label{tab:questions}
\small
\renewcommand{\arraystretch}{1.15}
\begin{tabularx}{\textwidth}{@{}>{\raggedright\arraybackslash}p{0.24\textwidth}>{\raggedright\arraybackslash}X>{\raggedright\arraybackslash}p{0.24\textwidth}@{}}
\toprule
Source & Question & Answer \\
\midrule
\cite[Problem~9]{Zhao2022Survey}; \cite[Problem~1]{Zhao2025Johnstone}
& Is $\J^{\N}$ $C_{\sigma}$-unique?
& Yes, for every power $\J^I$; \cref{thm:J-arbitrary-power}. \\
\cite[Problem~8]{Zhao2022Survey}
& Do binary products preserve $C_{\sigma}$-uniqueness?
& No, even with sober factors and sober product; \cref{thm:product-counterexample}. \\
\cite[Problems~5 and~6]{Zhao2022Survey}
& Is every complete lattice, or every sober countable complete lattice, $C_{\sigma}$-unique?
& No, even for a sober countable frame; \cref{thm:countable-frame-counterexample}. \\
\bottomrule
\end{tabularx}
\end{table}
\FloatBarrier

Two natural problems remain.

\begin{question}\label{q:locally-compact-sober}
Is every locally compact sober dcpo $C_{\sigma}$-unique?
\end{question}

\begin{question}
\label{q:product-stability}
What structural conditions on a family \((D_i)_{i\in I}\) of
\(C_{\sigma}\)-unique dcpos guarantee that
\(
  \prod_{i\in I}D_i
\)
is again \(C_{\sigma}\)-unique?
\end{question}

The first problem asks whether local compactness and sobriety suffice for reconstruction from Scott-closed sets. The second asks for a sharp boundary between product-stable and product-unstable classes. More generally, it would be useful to replace the present global obstruction by local criteria expressed through compact elements, way-below relations, or other finite approximation data.

\end{document}